\documentclass{amsart}

\usepackage{verbatim}
\usepackage{amssymb}
\usepackage{amscd}
\usepackage{amsmath}
\newtheorem{theorem}{Theorem}[section]

\newtheorem{cor}[theorem]{Corollary}
\newtheorem{lemma}[theorem]{Lemma}
\newtheorem{proposition}[theorem]{Proposition}
\numberwithin{equation}{subsection}

\title{From a cotangent sum to a generalized\\ totient function}
\author{Michael Th. Rassias}
\date{\today}
\address{Institute of Mathematics, University of Zurich, CH-8057, Zurich, Switzerland
 \& Institute for Advanced Study, Program in Interdisciplinary Studies,
1 Einstein Dr, Princeton, NJ 08540, USA.}
\email{michail.rassias@math.uzh.ch, michailrassias@math.princeton.edu}

\begin{document}

 \maketitle
 
 \begin{abstract} In this paper we investigate a certain category of cotangent sums and more specifically the sum
 $$\sum_{m=1}^{b-1}\cot\left(\frac{\pi m}{b}\right)\sin^{3}\left(2\pi m\frac{a}{b}\right)\:$$
 and associate the distribution of its values to a generalized totient function $\phi(n,A,B)$, where
$$\phi(n,A,B):=\sum_{\substack{A\leq k \leq B \\ (n,k)=1}}1\:.$$ 
One of the methods used consists in the exploitation of relations between trigonometric sums and
the fractional part of a real number.  
\newline

\end{abstract}

\section{Introduction}
\vspace{5mm}
\noindent For $a$, $b$, $n\in\mathbb{N}$, let
$$x_n:=\left\{\frac{na}{b}\right\}=\frac{na}{b}-\left\lfloor \frac{na}{b}\right\rfloor,$$
where $\left\lfloor u\right\rfloor$ stands for the floor function of the real number $u$. In other words $x_n$ denotes the fractional part of the rational number $na/b$ (for an extensive study of
fractional parts of real numbers see \cite{Furdui}).\\ We know (see \cite{Ras}, Proposition 2.1) that
\begin{proposition}\label{x:arxiko}
For every $a$, $b$, $n\in\mathbb{N}$, $b\geq 2$, we have
$$\sum_{m=1}^{b-1}\cot\left(\frac{\pi m}{b}\right)\cos\left(2\pi mn\frac{a}{b}\right)=0\:.$$
If $b\not|na$ then we also have
$$x_n=\frac{1}{2}-\frac{1}{2b}\sum_{m=1}^{b-1}\cot\left(\frac{\pi m}{b}\right)\sin\left(2\pi mn\frac{a}{b}\right)\:.$$
\end{proposition}

Based on the  trigonometric identity 
$$\cos(n\theta)=\sum_{k=0}^{n}\cos^k(\theta)\: \sin^{n-k}(\theta)\:\cos\left(\frac{n-k}{2}\pi \right)  $$
in combination with the above proposition, we can inductively prove that for every $a$, $b$, $q$, $n\in\mathbb{N}$, $b\geq 2$, it holds
$$\sum_{m=1}^{b-1}\cot\left(\frac{\pi m}{b}\right)\cos^{q}\left(2\pi mn\frac{a}{b}\right)=0\:.$$
Additionally, by Proposition \ref{x:arxiko} we can prove that for every $a$, $b$, $n\in\mathbb{N}$, $b\geq 2$, we have
 $$\sum_{m=1}^{b-1}\cot\left(\frac{\pi m}{b}\right)\sin^{2}\left(2\pi mn\frac{a}{b}\right)=0\:.$$
 Hence, the natural question of calculating cotangent sums of the form
 $$\sum_{m=1}^{b-1}\cot\left(\frac{\pi m}{b}\right)\sin^{r}\left(2\pi mn\frac{a}{b}\right)\:,$$
 where $r\in\mathbb{N}$ and $r\geq 3$, arises.\\
 Interestingly, the investigation of the above category of cotangent sums, with $r\geq 3$, turns out to be more complex.\\ 
 In the subsequent sections, we shall calculate  the cotangent sum $S(1,a,b)$, where
 $$S(n,a,b)=\sum_{m=1}^{b-1}\cot\left(\frac{\pi m}{b}\right)\sin^{3}\left(2\pi mn\frac{a}{b}\right)\:$$
 and associate the distribution of its values to a generalized totient function $\phi(n,A,B)$,
 where
$$\phi(n,A,B):=\sum_{\substack{A\leq k \leq B \\ (n,k)=1}}1\:.$$ 
Moreover, we  prove several properties of $\phi(n,A,B)$ including an asymptotic formula. Namely, our main results are the following:
\begin{proposition}
Let $a$, $b\in\mathbb{N}$, where $(a,b)=1$, $a\geq b \geq 2$ and $b\neq 3$. Then
$$S(1,a,b)=0\ \text{or}\ \pm b/2\:.$$
\end{proposition} 
\begin{proposition}
Let $a$, $b\in\mathbb{N}$, $b\geq 2$, where $(a,b)=1$ and $b\neq 3$. Then
$$S(1,a,b)=0$$ 
if and only if $2b=3a+k+1$, for some $k$, $0\leq k \leq b-2$.
\end{proposition}
\begin{cor}Let $a$, $b\in\mathbb{N}$, $b\geq 2$, where $(a,b)=1$ and $b\neq 3$. Then, the number of integers $a$, such that $1\leq a \leq b-1$ and $S(1,a,b)=0$, is given by the following formula
$$\#\left\{a\:|\:S(1,a,b)=0\right\}=\phi\left(b,\left\lceil \frac{b+1}{3}\right\rceil,\left\lfloor \frac{2b-1}{3}\right\rfloor\right).$$
\end{cor}
\begin{proposition}
Let $a$, $b\in\mathbb{N}$, $b\geq 2$, where $(a,b)=1$ and $b\neq 3$. Then
$$S(1,a,b)=\frac{b}{2}$$ 
if and only if $b=3a+k+1$, for some $k$, $0\leq k \leq b-2$.
\end{proposition} 
\begin{cor}Let $a$, $b\in\mathbb{N}$, $b\geq 2$, where $(a,b)=1$ and $b\neq 3$. Then, the number of integers $a$, such that $1\leq a \leq b-1$ and $S(1,a,b)=b/2$, is given by the following formula
$$\#\left\{a\:|\:S(1,a,b)=\frac{b}{2}\right\}=\phi\left(b,1,\left\lfloor \frac{b-1}{3}\right\rfloor\right).$$
\end{cor}
\begin{proposition}
Let $a$, $b\in\mathbb{N}$, $b\geq 2$, where $(a,b)=1$ and $b\neq 3$. Then
$$S(1,a,b)=-\frac{b}{2}$$ 
if and only if $3b=3a+k+1$, for some $k$, $0\leq k \leq b-2$.
\end{proposition} 
\begin{cor}Let $a$, $b\in\mathbb{N}$, $b\geq 2$, where $(a,b)=1$ and $b\neq 3$. Then, the number of integers $a$, such that $1\leq a \leq b-1$ and $S(1,a,b)=-b/2$, is given by the following formula
$$\#\left\{a\:|\:S(1,a,b)=-\frac{b}{2}\right\}=\phi\left(b,\left\lceil \frac{2b+1}{3}\right\rceil,\left\lfloor \frac{3b-1}{3}\right\rfloor\right).$$
\end{cor}
\begin{proposition}
Let $n$, $A$, $B\in\mathbb{N}$, $n>1$. Then, we have
$$\phi(n,A,B)=\frac{B-A}{n}\phi(n)+\delta_{n,A}+O\left(\sum_{d|n}\mu(d)^2\right)\:,$$
where $\delta_{n,A}=1$ if $(n,A)=1$ and $0$ otherwise.
\end{proposition}
 \vspace{5mm}
\section{Preliminaries}
\vspace{5mm}
\begin{proposition}$\label{x:prop1}$
For every $a$, $b$, $n\in\mathbb{N}$, $b\geq 2$, such that $b\not|3n$, we have
$$x_{3n}=3x_n-1+\frac{2}{b}S(n,a,b)\:,$$
where $x_n:=\left\{na/b\right\}$ and
$$S(n,a,b):=\sum_{m=1}^{b-1}\cot\left(\frac{\pi m}{b}\right)\sin^3\left(2\pi mn\frac{a}{b}\right)\:.$$
\end{proposition}
\begin{proof}
We know that 
$$x_n=\frac{1}{2}-\frac{1}{2b}\sum_{m=1}^{b-1}\cot\left(\frac{\pi m}{b}\right)\sin\left(2\pi mn\frac{a}{b}\right)\:,$$
for every $b\not|n$. So, we get
$$x_{3n}=\frac{1}{2}-\frac{1}{2b}\sum_{m=1}^{b-1}\cot\left(\frac{\pi m}{b}\right)\sin\left(2\pi m(3n)\frac{a}{b}\right)\:,$$
for every $b\not|3n$. Thus
$$x_{3n}=\frac{1}{2}-\frac{1}{2b}\sum_{m=1}^{b-1}\cot\left(\frac{\pi m}{b}\right)\left(3\sin\left(2\pi mn\frac{a}{b}\right)-4\sin^3\left(2\pi mn\frac{a}{b}\right)\right)\:,$$
for every $b\not|3n$. So
\begin{eqnarray}
x_{3n}&=&\frac{1}{2}-\frac{3}{2b}\sum_{m=1}^{b-1}\cot\left(\frac{\pi m}{b}\right)\sin\left(2\pi mn\frac{a}{b}\right)+\frac{4}{2b}\sum_{m=1}^{b-1}\cot\left(\frac{\pi m}{b}\right)\sin^3\left(2\pi mn\frac{a}{b}\right)\nonumber\\
&=&3\left(\frac{1}{2}-\frac{1}{2b}\sum_{m=1}^{b-1}\cot\left(\frac{\pi m}{b}\right)\sin\left(2\pi mn\frac{a}{b}\right)\right)-1+\frac{2}{b}\sum_{m=1}^{b-1}\cot\left(\frac{\pi m}{b}\right)\sin^3\left(2\pi mn\frac{a}{b}\right)\:,\nonumber
\end{eqnarray}
for every $b\not|3n$. Hence
$$x_{3n}=3x_n-1+\frac{2}{b}S(n,a,b)\:,$$
for every $b\not|3n$.
\end{proof}
\begin{lemma}$\label{x:lem2}$
For every $a$, $b$, $n\in\mathbb{N}$, $b\geq 2$ and every $k\in\mathbb{N}\cup{\left\{0\right\}}$, we have
$$\left\{\frac{na+k}{b}\right\}=x_n+\frac{k}{b}-\frac{1}{b}E(n,k)\:,$$
where
$$E(n,k):=\sum_{\lambda=na+1}^{na+k}\sum_{m=0}^{b-1}e^{2\pi im\lambda/b}\:,$$
for $k\in\mathbb{N}$ and $E(n,0):=0$.
\end{lemma}
\begin{proof}
We know (see \cite{Ras}, Section 2) that 
$$\left\{\frac{a}{b}\right\}=\frac{a}{b}-\frac{1}{b}\sum_{\lambda=1}^{a}\sum_{m=0}^{b-1}e^{2\pi im\lambda/b}\:.$$
Thus, we can write
\begin{eqnarray}
\left\{\frac{na+k}{b}\right\}&=&\frac{na+k}{b}-\frac{1}{b}\sum_{\lambda=1}^{na+k}\sum_{m=0}^{b-1}e^{2\pi im\lambda/b}\nonumber\\
&=&\frac{na}{b}+\frac{k}{b}-\frac{1}{b}\sum_{\lambda=1}^{na}\sum_{m=0}^{b-1}e^{2\pi im\lambda/b}-\frac{1}{b}\sum_{\lambda=na+1}^{na+k}\sum_{m=0}^{b-1}e^{2\pi im\lambda/b}\nonumber\\
&=&\left(\frac{na}{b}-\frac{1}{b}\sum_{\lambda=1}^{na}\sum_{m=0}^{b-1}e^{2\pi im\lambda/b}\right)+\left(\frac{k}{b}-\frac{1}{b}\sum_{\lambda=na+1}^{na+k}\sum_{m=0}^{b-1}e^{2\pi im\lambda/b}\right)\nonumber\\
&=&\left\{\frac{na}{b}\right\}+\frac{k}{b}-\frac{1}{b}\sum_{\lambda=na+1}^{na+k}\sum_{m=0}^{b-1}e^{2\pi im\lambda/b}\:,\nonumber
\end{eqnarray}
for every $k\in\mathbb{N}$.\\
For the case when $k=0$, the result is clear.

\end{proof}
The following proposition also holds.
\begin{proposition}$\label{x:lem2}$
Let $a$, $b\in\mathbb{N}$, with $b\geq 2$.\\
If $a\not\equiv0\;(\bmod\;b)$, then
$$\left\{\frac{a}{b}\right\}=\left\{\frac{a-1}{b}\right\}+\frac{1}{b}\:.$$
If $a\equiv0\;(\bmod\;b)$, then
$$\left\{\frac{a-2}{b}\right\}=1-\frac{2}{b}\:.$$
\end{proposition}
\begin{proposition}$\label{x:vasiko}$
Let $a$, $b\in\mathbb{N}$, where $(a,b)=1$, $a\geq b \geq 2$ and $b\neq 3$. Then
$$(3\nu+2)b=(3a+k+1)+3E(1,k)+2S(1,a,b)\:,$$
for some integer $k$, $0\leq k\leq b-2$, such that $3a+k+1\equiv 0\:(\bmod\;b)\:,$ where
$$\nu:=\left\lfloor \frac{a+k}{b}\right\rfloor$$
and
$$S(n,a,b)=\sum_{m=1}^{b-1}\cot\left(\frac{\pi m}{b}\right)\sin^3\left(2\pi mn\frac{a}{b}\right)\:.$$
\end{proposition}
\begin{proof}
Since $b\neq 3$ and $(a,b)=1$, it is clear that $b$ does not divide $3a$. Thus, $b$ should divide one of the consecutive integers
$$3a+1,\;3a+2,\;\ldots,\;3a+(b-1)\:.$$
In other words, there exists $k$, with $0\leq k\leq b-2$, such that 
$$3a+k+1\equiv 0\:(\bmod\;b)\:.$$
But, then it is obvious that $b\not|3a+k$. Hence, by Proposition 5.2 of \cite{Ras}, we get
\[
\left\{\frac{3a+k}{b}\right\}=\left\{\frac{3a+k-1}{b}\right\}+\frac{1}{b}\:.\tag{1}
\]
Also, since $3a+k+1\equiv 0\:(\bmod\;b)$, by Proposition 5.2 of \cite{Ras}, it follows that
$$\left\{\frac{3a+k-1}{b}\right\}=1-\frac{2}{b}\:.$$
So, by the above relation and (1), we obtain
\[
\left\{\frac{3a+k}{b}\right\}=1-\frac{1}{b}\:.\tag{2}
\]
However, by Lemma $\ref{x:lem2}$ we know that 
$$\left\{\frac{na+k}{b}\right\}=x_n+\frac{k}{b}-\frac{1}{b}E(n,k)\:,$$
for every $n\in\mathbb{N}$. Thus, this yields
$$\left\{\frac{3a+k}{b}\right\}=x_3+\frac{k}{b}-\frac{1}{b}E(3,k)$$
and
$$\left\{\frac{a+k}{b}\right\}=x_1+\frac{k}{b}-\frac{1}{b}E(1,k)\:.$$
Therefore,
\[
x_3=\left\{\frac{3a+k}{b}\right\}-\frac{k}{b}+\frac{1}{b}E(3,k)\:\tag{3.1}
\]
and
\[
x_1=\left\{\frac{a+k}{b}\right\}-\frac{k}{b}+\frac{1}{b}E(1,k)\:.\tag{3.2}
\]
But $E(3,k)=0$, since we have assumed that $3a+k+1\equiv0\;(\bmod\;b)$. Thus, $3a+j\not\equiv0\;(\bmod\;b)$, $1\leq j\leq k$. Otherwise $b|(3a+k+1)-(3a+j)\Rightarrow b|(k-j)+1$ and for $k\geq1$ it holds $0\leq k-j\leq k-1\leq b-3 \Rightarrow 1\leq(k-j)+1\leq b-2$. If $k=0$ then by definition $E(3,0)=0$.\\
\noindent In addition, by Proposition $\ref{x:prop1}$ we know that 
$$x_3=3x_1-1+\frac{2}{b}S(1,a,b)\:.$$
Hence, by relations (3.1), (3.2), we obtain
$$\left\{\frac{3a+k}{b}\right\}-\frac{k}{b}+\frac{1}{b}E(3,k)=3\left\{\frac{a+k}{b}\right\}-\frac{3k}{b}+\frac{3}{b}E(1,k)-1+\frac{2}{b}S(1,a,b)$$
or
$$\left\{\frac{3a+k}{b}\right\}=3\left\{\frac{a+k}{b}\right\}-\frac{2k}{b}+\frac{3}{b}E(1,k)-1+\frac{2}{b}S(1,a,b).$$
Hence, by (2) we get
\[
1-\frac{1}{b}=3\left\{\frac{a+k}{b}\right\}-\frac{2k}{b}+\frac{3}{b}E(1,k)-1+\frac{2}{b}S(1,a,b).\tag{4}
\]
But since $a\geq b$, it is clear that
$$\left\lfloor \frac{a+k}{b}\right\rfloor=\nu,\ \ \nu\in\mathbb{N}\:.$$
Therefore, by (4) we get\\
\begin{eqnarray}
1-\frac{1}{b}&=&3\left(\frac{a+k}{b}\right)-3\nu-\frac{2k}{b}+\frac{3}{b}E(1,k)-1+\frac{2}{b}S(1,a,b)\nonumber\\
&=&3\frac{a}{b}+\frac{k}{b}-(3\nu+1)+\frac{3}{b}E(1,k)+\frac{2}{b}S(1,a,b)\:.\nonumber
\end{eqnarray}
Thus
$$\frac{a}{b}=\frac{3\nu+1}{3}-\frac{k+1}{3b}-\frac{1}{b}E(1,k)-\frac{2}{3b}S(1,a,b)+\frac{1}{3}\:$$
or
$$3a=(3\nu+1)b-k-1-3E(1,k)-2S(1,a,b)+b\:$$
or
$$(3\nu+2)b=(3a+k+1)+3E(1,k)+2S(1,a,b)\:.$$
\end{proof}
\begin{cor}
Let $a$, $b\in\mathbb{N}$, where $(a,b)=1$, $a\geq b \geq 2$ and $b$ is even. Then
$$2S(1,a,b)\equiv0\:(\bmod\; b)$$
and therefore $S(1,a,b)$ is an integer.
\end{cor}
\begin{proof}
We know that $3a+k+1\equiv 0\;(\bmod\;b)$. Also, $E(1,k)\equiv 0\;(\bmod\;b)$, since its terms are either $b$ or $0$. Hence,
$$2S(1,a,b)\equiv 0\;(\bmod\;b)\:.$$
Since $b$ is an even integer, it follows that $S(1,a,b)\in\mathbb{Z}$.\\
\end{proof}
%
%
%
\section{Computing the values of $S(1,a,b)$}
\vspace{5mm}
By Proposition $\ref{x:prop1}$ and since $0\leq x_3<1$, it easily follows that 
$$\left|S(1,a,b)\right|<b\:.$$
The above inequality and the fact that $S(1,a,b)$ is always an integer when $b$ is even, lead us to the assumption that the values of this cotangent sum could possibly be very specific. Some numerical experiments revealed that the value of $S(1,a,b)$ was either $0$ or $\pm b/2$. Hence, with some further investigation we obtained the following result.
\begin{proposition}$\label{x:times}$
Let $a$, $b\in\mathbb{N}$, where $(a,b)=1$, $a\geq b \geq 2$ and $b\neq 3$. Then
$$S(1,a,b)=0\ \text{or}\ \pm b/2\:.$$
\end{proposition} 
\begin{proof}
By Proposition $\ref{x:vasiko}$ we know that
\[
2S(1,a,b)=(3\nu+2)b-(3a+k+1)-3E(1,k)\:,\tag{5}
\]
We can consider $a$, such that $1\leq a\leq b-1$ due to the periodicity of $S(1,a,b)$ with period $a$. Thus, since $0\leq k\leq b-2$ we get
$$\frac{1}{b}\leq \frac{a+k}{b}\leq2-\frac{3}{b}\:.$$
Therefore
$$\left\lfloor \frac{a+k}{b}\right\rfloor=0\ \text{or}\ 1\:.$$
In other words, $\nu=0$ or $1$. Hence, we can consider the following cases.\\
\textit{Case 1.} If $\nu=0$, we have
$$2S(1,a,b)=2b-(3a+k+1)-3E(1,k)\:.$$
That is
$$S(1,a,b)=b-\frac{3a+k+1}{2}-\frac{3E(1,k)}{2}\:.$$
Set
$$\frac{3a+k+1}{2}+\frac{3E(1,k)}{2}:=m\in\mathbb{Q}^+\:.$$
Then, we can write
$$S(1,a,b)=b-m\:.$$
However, we know that $\left|S(1,a,b)\right|<b\:$ and thus
$$\left|m-b\right|<b$$
or
\[
0<m<2b\:.\tag{6}
\]
But, since both $3a+k+1$ and $3E(1,k)$ are divisible by $b$, it follows that $2m$ is divisible by $b$. Therefore, we obtain
$$\frac{2m}{b}=r,\  \text{where $r\in\mathbb{N}$}$$
or equivalently
\[
m=\frac{b}{2}\cdot r \tag{7}
\]
By (6), (7) it follows that the only possible values for $m$ are
$$m=\frac{b}{2},\:b,\:\frac{3b}{2}\:.$$
Consequently, the only possible values that $S(1,a,b)$ may obtain, in the case when $\nu=0$, are
$$S(1,a,b)=0,\:\pm \frac{b}{2}\:.$$
\textit{Case 2.} If $\nu=1$, by (5) we have
$$2S(1,a,b)=5b-(3a+k+1)-3E(1,k)\:$$
or equivalently
\[
S(1,a,b)=\frac{5b}{2}-m\:,\tag{8}
\]
where $m$ is defined as in Case 1. Thus, similarly to the case when $\nu=0$, we get
$$\left|m-\frac{5b}{2}\right|<b\:,$$
from which it follows that
$$\frac{3b}{2}<m<\frac{7b}{2}\:,\ m\in\mathbb{N}\:.$$
Additionally, we have
$$m=\frac{b}{2}\cdot r,\:r\in\mathbb{N}\:.$$
Hence, the possible values of $m$ are
$$m=2b,\:\frac{5b}{2},\:3b\:.$$
Therefore, by (8) it follows that the only possible values that $S(1,a,b)$ may obtain, in the case when $\nu=1$, are
$$S(1,a,b)=0,\:\pm \frac{b}{2}\:.$$

\end{proof}
Now that we have specified the only values which the cotangent sum $S(1,a,b)$ can obtain, an interesting question is to investigate when does this sum obtain these values. Thus, in the following we will determine the values of the integer $a$, for fixed $b$, for which $S(1,a,b)=0,\:\pm\: b/2$, respectively.
\section{The distribution of the values of $S(1,a,b)$}
\vspace{5mm}
\noindent \textbf{The set of integer values $a$ for which $S(1,a,b)=0$.}\\ \\
\noindent By Proposition $\ref{x:vasiko}$, for $S(1,a,b)=0$ we obtain
\[
(3\nu+2)b=(3a+k+1)+3E(1,k)\:.\tag{9}
\]
As we have illustrated in the previous sections, $\nu=0$ or $1$ and $E(1,k)=0$ or $b$. Thus, we can distinguish the following cases.\\
\textit{Case 1} If $\nu=0$, by (9) we get
$$2b=(3a+k+1)+3E(1,k)\:.$$
Hence, if $E(1,k)=0$ then $2b=3a+k+1$. On the other hand, if $E(1,k)=b$ then $3a+k+1=-b<0$, which is a contradiction.\\
\noindent \textit{Case 2.} If $\nu=1$, by (9) we obtain
$$5b=(3a+k+1)+3E(1,k)\:.$$
Thus, if $E(1,k)=0$ then $5b=3a+k+1$. But, since $1\leq a\leq b-1$ and $0\leq k \leq b-2$, it follows that $4\leq 3a+k+1\leq 4b-4$, which is a contradiction. If $E(1,k)=b$ then $2b=3a+k+1$.\\
Therefore, we obtain the following proposition
\begin{proposition}$\label{x:0}$
Let $a$, $b\in\mathbb{N}$, $b\geq 2$, where $(a,b)=1$ and $b\neq 3$. Then
$$S(1,a,b)=0$$ 
if and only if $2b=3a+k+1$, for some $k$, $0\leq k \leq b-2$.
\end{proposition}
\noindent By the above proposition it follows that the only values of $a$ which can be zeros of $S(1,a,b)$ are the ones for which $(a,b)=1$ and
$$\left\lceil \frac{b+1}{3}\right\rceil \leq a \leq \left\lfloor \frac{2b-1}{3}\right\rfloor\:.$$
Hence, we obtain the following corollary.
\begin{cor}Let $a$, $b\in\mathbb{N}$, $b\geq 2$, where $(a,b)=1$ and $b\neq 3$. Then, the number of integers $a$, such that $1\leq a \leq b-1$ and $S(1,a,b)=0$, is given by the following formula
$$\#\left\{a\:|\:S(1,a,b)=0\right\}=\phi\left(b,\left\lceil \frac{b+1}{3}\right\rceil,\left\lfloor \frac{2b-1}{3}\right\rfloor\right),$$
where
$$\phi(n,A,B)=\sum_{\substack{A\leq k \leq B \\ (n,k)=1}}1\:.$$
\end{cor} 
\vspace{5mm}
\noindent \textbf{The set of integer values $a$ for which $S(1,a,b)=b/2$.}\\ \\
We shall now investigate the case when $S(1,a,b)=-b/2$, $1\leq a \leq b-1$, $(a,b)=1$. More specifically, by Proposition $\ref{x:vasiko}$ we obtain
$$(3\nu+2)b=(3a+k+1)+3E(1,k)+b\:.$$
\textit{Case 1.} If $\nu=0$ we have
$$b=(3a+k+1)+3E(1,k)\:.$$
Thus, if $E(1,k)=0$ then $b=3a+k+1$. If $E(1,k)=b$ then $3a+k+1=-2b<0$, which is a contradiction. \\
\textit{Case 2.} If $\nu=1$ we have
$$4b=(3a+k+1)+3E(1,k)\:.$$
Thus, if $E(1,k)=0$ then $4b=3a+k+1\leq 4b-4$ which is a contradiction. If $E(1,k)=b$ then $b=3a+k+1$.
Therefore, from the above we obtain the following proposition.
\begin{proposition}$\label{x:b/2}$
Let $a$, $b\in\mathbb{N}$, $b\geq 2$, where $(a,b)=1$ and $b\neq 3$. Then
$$S(1,a,b)=\frac{b}{2}$$ 
if and only if $b=3a+k+1$, for some $k$, $0\leq k \leq b-2$.
\end{proposition} 
\noindent By the above proposition it follows that the only values of $a$ for which \mbox{$S(1,a,b)=b/2$} are the ones for which $(a,b)=1$ and
$$1 \leq a \leq \left\lfloor \frac{b-1}{3}\right\rfloor\:.$$
Hence, we obtain the following corollary.
\begin{cor}Let $a$, $b\in\mathbb{N}$, $b\geq 2$, where $(a,b)=1$ and $b\neq 3$. Then, the number of integers $a$, such that $1\leq a \leq b-1$ and $S(1,a,b)=b/2$, is given by the following formula
$$\#\left\{a\:|\:S(1,a,b)=\frac{b}{2}\right\}=\phi\left(b,1,\left\lfloor \frac{b-1}{3}\right\rfloor\right).$$
\end{cor}
\vspace{5mm} 
\noindent \textbf{The set of integer values $a$ for which $S(1,a,b)=-b/2$.}\\ \\
We shall now investigate the final case when $S(1,a,b)=-b/2$, $1\leq a \leq b-1$, $(a,b)=1$. Again by Proposition $\ref{x:vasiko}$ we obtain
$$(3\nu+2)b=(3a+k+1)+3E(1,k)-b\:.$$
\textit{Case 1.} If $\nu=0$ we have
$$3b=(3a+k+1)+3E(1,k)\:.$$
Thus, if $E(1,k)=0$ then $3b=3a+k+1$. If $E(1,k)=b$ then $3a+k+1=0$, which is a contradiction. \\
\textit{Case 2.} If $\nu=1$ we have
$$6b=(3a+k+1)+3E(1,k)\:.$$
Thus, if $E(1,k)=0$ then $6b=3a+k+1\leq 4b-4$ from which we get $2b\leq -4$ which is a contradiction. If $E(1,k)=b$ then $3b=3a+k+1$.
Therefore, from the above we obtain the following proposition.
\begin{proposition}$\label{x:-b/2}$
Let $a$, $b\in\mathbb{N}$, $b\geq 2$, where $(a,b)=1$ and $b\neq 3$. Then
$$S(1,a,b)=-\frac{b}{2}$$ 
if and only if $3b=3a+k+1$, for some $k$, $0\leq k \leq b-2$.
\end{proposition} 
\noindent By the above proposition it follows that the only values of $a$ for which $S(1,a,b)=-b/2$ are the ones for which $(a,b)=1$ and
$$\left\lceil \frac{2b+1}{3}\right\rceil \leq a \leq \left\lfloor \frac{3b-1}{3}\right\rfloor\:.$$
Hence, we obtain the following corollary.
\begin{cor}Let $a$, $b\in\mathbb{N}$, $b\geq 2$, where $(a,b)=1$ and $b\neq 3$. Then, the number of integers $a$, such that $1\leq a \leq b-1$ and $S(1,a,b)=-b/2$, is given by the following formula
$$\#\left\{a\:|\:S(1,a,b)=-\frac{b}{2}\right\}=\phi\left(b,\left\lceil \frac{2b+1}{3}\right\rceil,\left\lfloor \frac{3b-1}{3}\right\rfloor\right).$$
\end{cor}

%
\vspace{10mm}
\section{The function $\phi(n,A,B)$}
\vspace{5mm}
\begin{lemma}$\label{x:lima}$
Let $A,B\in\mathbb{N}$ and
$$\phi(n,A,B):=\sum_{\substack{A\leq k \leq B \\ (n,k)=1}}1\:.$$
Then, we have 
$$\phi(n,A,B)=\sum_{d|n}\mu(d)\left(\left\lfloor \frac{B}{d}\right\rfloor-\left\lceil \frac{A}{d}\right\rceil\right)\:,$$
where $\mu(n)$ is the M\"{o}bius function. 
\end{lemma}
\begin{proof}
We know that 
$$\sum_{d|N}\mu(d)=\left\lfloor \frac{1}{N}\right\rfloor\:.$$
Thus, we get
$$\phi(n,A,B)=\sum_{k=A}^{B}\:\sum_{d|(n,k)}\mu(d)=\sum_{k=A}^{B}\sum_{\substack{d|n \\ d|k}}\mu(d)\:.$$
Hence, $k=md$, for some $m\in\mathbb{N}$. But, since $A\leq k\leq B$ it follows that
$$\frac{A}{d}\leq m\leq \frac{B}{d}\:.$$
Therefore, we obtain
\begin{eqnarray}
\phi(n,A,B)&=&\sum_{d|n}\:\sum_{m=\left\lceil A/d\right\rceil}^{\left\lfloor B/d\right\rfloor}\mu(d)\nonumber\\
&=&\sum_{d|n}\mu(d)\sum_{m=\left\lceil A/d\right\rceil}^{\left\lfloor B/d\right\rfloor}1\nonumber\\
&=&\sum_{d|n}\mu(d)\left(\left\lfloor \frac{B}{d}\right\rfloor-\left\lceil \frac{A}{d}\right\rceil\right)\:.\nonumber
\end{eqnarray}

\end{proof}
\begin{proposition}$\label{x:kedriko}$
Let $n$, $A$, $B\in\mathbb{N}$, $n>1$. Then, we have
$$\phi(n,A,B)=\frac{B-A}{n}\phi(n)+\delta_{n,A}+O\left(\sum_{d|n}\mu(d)^2\right)\:,$$
where $\delta_{n,A}=1$ if $(n,A)=1$ and $0$ otherwise.
\end{proposition}
\begin{proof}
By Lemma $\ref{x:lima}$ we get
\begin{eqnarray}
\phi(n,A,B)&=&\sum_{d|n}\mu(d)\left\lfloor \frac{B}{d}\right\rfloor-\sum_{d|n}\mu(d)\left\lceil \frac{A}{d}\right\rceil.\nonumber
\end{eqnarray}
However, since 
\begin{equation}
\left\lceil \frac{A}{d}\right\rceil=\left\{
\begin{array}{l l}
    \left\lfloor \frac{A}{d}\right\rfloor+1\:, & \quad \text{if $d\not| A$}\vspace{2mm}\\ 
    \left\lfloor \frac{A}{d}\right\rfloor \:, & \quad \text{otherwise}\:,\\
  \end{array} \right.
  \nonumber
\end{equation}
it follows that
\begin{eqnarray}
\phi(n,A,B)&=&\sum_{d|n}\mu(d)\left\lfloor \frac{B}{d}\right\rfloor-\sum_{d|n}\mu(d)\left( \left\lfloor \frac{A}{d}\right\rfloor+1\right)+\sum_{\substack{d|n \\ d|A}}\mu(d)\nonumber\\
&=&\phi(n,1,B)-\phi(n,1,A)-\sum_{d|n}\mu(d)+\sum_{\substack{d|n \\ d|A}}\mu(d).\nonumber
\end{eqnarray}
However, it is a well known fact that for $n>1$ it holds $\sum_{d|n}\mu(d)=0$. Additionally, one can easily show that 
\begin{equation}
\sum_{\substack{d|n \\ d|A}}\mu(d)=\left\{
\begin{array}{l l}
    1\:, & \quad \text{if $(n,A)=1$}\vspace{2mm}\\ 
    0\:, & \quad \text{otherwise}\:.\\
  \end{array} \right.
  \nonumber
\end{equation}
Therefore, we obtain
\[
\phi(n,A,B)=\phi(n,1,B)-\phi(n,1,A)+\delta_{n,A}\:.\tag{10}
\]
The function $\phi(n,1,x)$, $x\in\mathbb{R}^+$ is exactly the so-called Legendre totient function . Generally, we have
\begin{eqnarray}
\phi(n,1,x)&=&\sum_{d|n}\mu(d)\left\lfloor \frac{x}{d}\right\rfloor\nonumber=\sum_{d|n}\mu(d)\left(\frac{x}{d}+O(1)\right)\nonumber\\
&=&x\sum_{d|n}\frac{\mu(d)}{d}+O\left(\sum_{d|n}\mu(d)^2\right)\nonumber\\
&=&x\frac{\phi(n)}{n}+O\left(\sum_{d|n}\mu(d)^2\right).\nonumber
\end{eqnarray}
Hence, by (10) we obtain the desired result.

\end{proof}

\noindent The above proposition presents an approximation formula for the generalized totient function $\phi(n,A,B)$ up to the error
$$\sum_{d|n}\mu(d)^2=2^{\omega(n)}.$$
However, it is a known fact that for every positive integer $n$ and every $\epsilon>0$, we have
$$2^{\omega(n)}\leq d(n)\ll_{\epsilon}n^{\epsilon},$$
where $d(n)$ denotes the number of positive divisors of $n$ (for relevant properties of $d(n)$ cf. \cite{apostol}).\\
This demonstrates that the error term in Proposition $\ref{x:kedriko}$ is relatively small. 

Based just on the definition of the function $\phi(n,A,B)$ we can also prove the following two propositions.
\begin{proposition}
For every $n$, $A$, $B\in\mathbb{N}$, we have
$$\sum_{d|n}\phi\left(d,\frac{A}{d},\frac{B}{d}\right)=B-A+1\:.$$
\end{proposition} 
\begin{proof}
We consider the sets
$$N(A,B):=\{A,A+1,\ldots,B-1,B  \}$$
and
$$R(n,d\:;A,B):=\{ m\: : \: (m,n)=d, A\leq m\leq B \}\:.$$
It is evident that each set $R(n,d\:;A,B)$ is a subset of $N(A,B)$, containing those elements which have greatest common divisor $d$ with $n$. Since the sets $R(n,d\:;A,B)$ are mutually disjoint for different values of $d$, it follows that
\[
\sum_{d|n}|R(n,d\:;A,B)|=B-A+1.\tag{11}
\]
However, since $(m,n)=d$ is equivalent to $(m/d,n/d)=1$ and the inequality $A\leq m \leq B$ is equivalent to $A/d\leq m/d \leq B/d$, by setting $r:=m/d$ it follows that
$$|R(n,d\:;A,B)|=\:\vline\left\{r\: : \: \left( r,\frac{n}{d} \right)=1, \frac{A}{d}\leq r \leq \frac{B}{d}  \right\}\vline=\phi\left(\frac{n}{d},\frac{A}{d},\frac{B}{d} \right)\:.$$
Therefore, we obtain
$$\sum_{d|n}|R(n,d\:;A,B)|=\sum_{d|n}\phi\left(\frac{n}{d},\frac{A}{d},\frac{B}{d} \right)=\sum_{d|n}\phi\left(d,\frac{A}{d},\frac{B}{d} \right)$$
and hence, by (11) the desired result follows.

\end{proof}

\begin{proposition}
For every $n$, $A$, $B\in\mathbb{N}$, we have
$$\sum_{\substack{A\leq k\leq B \\ (k,n)=1}}k=\frac{n}{2}\:\phi(n,A,B).$$
\end{proposition}
\begin{proof}
Let $k_1,k_2,\ldots,k_{\phi(n,A,B)}$ be the integers such that $A\leq k_i\leq B$, $(k_i,n)=1$. Since $(k_i,n)=1$ is equivalent to $(n-k_i,n)=1$, it is evident that
\begin{eqnarray}
k_1+k_2+\cdots+k_{\phi(n,A,B)}&=&(n-k_1)+(n-k_2)+\cdots+(n-k_{\phi(n,A,B)})\nonumber\\
&=&n\phi(n,A,B)-(k_1+k_2+\cdots+k_{\phi(n,A,B)}),\nonumber
\end{eqnarray}
from which the desired result follows.
\end{proof}

\vspace{5mm}
\noindent\textbf{Acknowledgments.} The author would like to acknowledge financial support obtained through the Forschungskredit grant (Grant Nr. FK-15-106) of the University of Zurich. 


\begin{thebibliography}{99}%
\bibitem{apostol} T. M. Apostol, \textit{Introduction to Analytic Number Theory}, Springer--Verlag, New York, 1984.
\bibitem{Furdui} O. Furdui, \textit{Limits, Series, and Fractional Part Integrals - Problems in Mathematical Analysis}, Springer, New York, 2013.
\bibitem{Ras} M. Th. Rassias, \textit{A cotangent sum related to the zeros of the Estermann zeta function},  Applied Mathematics and Computation, 240(2014), 161-167.
\end{thebibliography}
\end{document}